\documentclass[12pt]{amsart}
\usepackage[alphabetic]{amsrefs}
\usepackage{mathdots, blkarray}
\usepackage{multicol}
\usepackage{graphicx}
\usepackage{amscd}
\usepackage{upgreek}
\usepackage{stmaryrd}
\usepackage{longtable}
\usepackage[T1]{fontenc}
\usepackage{latexsym, amsmath, amssymb, amsthm}
\usepackage[svgnames]{xcolor}
\usepackage{rsfso}
\usepackage[mathscr]{eucal}
\usepackage{mathtools}
\usepackage{mathptmx}
\usepackage{titletoc}
\usepackage{wrapfig}
\usepackage{float}
\usepackage{xypic}
\usepackage{microtype}
\usepackage{dsfont}
\usepackage{xcolor}
\usepackage{color}
\usepackage[colorlinks = true,
            linkcolor  = DarkBlue,
            urlcolor   = DarkRed,
            citecolor  = DarkGreen]{hyperref}
\usepackage{hyperref}
\allowdisplaybreaks	
\usepackage[centering, includeheadfoot, hmargin=1.0in, tmargin=0.5in, 
  bmargin=0.6in, headheight=6pt]{geometry}
\newtheorem{theorem}{Theorem}[section]
\newtheorem{prop}[theorem]{Proposition}

\newtheorem{lem}[theorem]{Lemma}

\newtheorem{thm}[theorem]{Theorem}

\newtheorem{rem}[theorem]{\rm\textsc{Remark}}
\newtheorem{exam}[theorem]{\rm\textsc{Example}}

\newcommand{\ideal}[1]{\ensuremath{\left\langle #1 \right\rangle}}

\newcommand{\oeq}[1]{\ensuremath{\overset{(\ref{#1})}{=}}}
\newcommand{\bslash}{\kern-0.1em\texttt{\scalebox{0.6}[1]{/}}\kern-0.15em \texttt{\scalebox{0.6}[1]{/}}}

\DeclareMathOperator{\ad}{ad}

\DeclareMathOperator{\dia}{diag}
	
\DeclareMathOperator{\SL}{SL}
\DeclareMathOperator{\Sp}{Sp}

\DeclareMathOperator{\trace}{tr}
\DeclareMathOperator{\GL}{GL}
\DeclareMathOperator{\cha}{char}

\DeclareMathOperator{\LM}{LM}
\DeclareMathOperator{\LC}{LC}
\DeclareMathOperator{\spol}{spol}
\DeclareMathOperator{\lcm}{lcm}

\newcommand{\A}{\mathcal{A}} 
\newcommand{\LL}{\mathcal{L}} 
\newcommand{\B}{\mathcal{B}} 
\newcommand{\D}{\mathcal{D}} 
\newcommand{\C}{\mathbb{C}}

\newcommand{\M}{\mathcal{M}} 
 
\newcommand{\R}{\mathbb{R}} 
\newcommand{\K}{\mathbb{K}} 
\newcommand{\F}{\mathbb{F}} 
\newcommand{\p}{\mathfrak{p}} 
\newcommand{\LA}{\Longleftarrow} 
\newcommand{\RA}{\Longrightarrow} 
 
\newcommand{\ra}{\longrightarrow}

\newcommand{\hbo}{$\hfill\Diamond$}

\begin{document}
\title{Geometry and classifications of some $\upomega$-Lie algebras} 
\def\shorttitle{Geometry and classifications of some $\upomega$-Lie algebras}

\author{Yin Chen}
\address{School of Mathematics and Physics, Key Laboratory of ECOFM of 
Jiangxi Education Institute, Jinggangshan University,
Ji'an 343009, Jiangxi, China \& Department of Finance and Management Science, University of Saskatchewan, Saskatoon, SK, Canada, S7N 5A7}
\email{yin.chen@usask.ca}

\author{Shan Ren}
\address{School of Mathematics and Statistics, Northeast Normal University, Changchun 130024, China}
\email{rens734@nenu.edu.cn}

\author{Runxuan Zhang}
\address{Department of Mathematics and Information Technology, Concordia University of Edmonton, Edmonton, AB, Canada, T5B 4E4}
\email{runxuan.zhang@concordia.ab.ca}

\begin{abstract}
Using group actions and orbit–stabilizer methods, we study the geometry of isomorphism classes of finite-dimensional $\upomega$-Lie algebras over a field $\mathbb{K}$ of characteristic $\neq 2$ and establish a one-to-one correspondence between the set of isomorphism classes and the orbit space of a stabilizer of $\upomega$. We also apply techniques from computational ideal theory to explore the geometric structure of the affine variety of all 3-dimensional $\upomega$-Lie algebras over $\mathbb{K}$, showing that this variety is a 6-dimensional irreducible affine variety and a complete intersection.  As an application, we derive a complete classification of all 3-dimensional $\upomega$-Lie algebras over an algebraically closed field of characteristic $\neq 2$, up to $\upomega$-Lie algebra isomorphism. 
\end{abstract}

\date{\today}
\maketitle \baselineskip=16.5pt

\dottedcontents{section}[1.16cm]{}{1.8em}{5pt}
\dottedcontents{subsection}[2.00cm]{}{2.7em}{5pt}

\section{Introduction}
\setcounter{equation}{0}
\renewcommand{\theequation}
{1.\arabic{equation}}
\setcounter{theorem}{0}
\renewcommand{\thetheorem}
{1.\arabic{theorem}}

\noindent Let $\K$ be any field of characteristic $\neq 2$ and $(L,[-,-])$ be a finite-dimensional skew-symmetric nonassociative algebra over $\K$. Let $\upomega:L\oplus L\ra \K$ be a bilinear form on $L$. The triple $(L,[-,-],\upomega)$ is called an \textit{$\upomega$-Lie algebra} over $\K$ if 
\begin{equation}\tag{$\upomega$-Jacobi identity}\label{Jacobi}
[[x,y],z]+[[y,z],x]+[[z,x],y]=\upomega(x,y)z+\upomega(y,z)x+\upomega(z,x)y
\end{equation}
for all $x,y,z\in L$. The notion of $\upomega$-Lie algebras, 
introduced in \cite{Nur07} as a natural generalization of Lie algebras ($\upomega=0$), was originally motivated by the study of symmetric 3-tensors on real vector spaces endowed with a Riemannian metric structure that arise in the context of isoparametric hypersurfaces in spheres as explored in \cite{BN07}. 
Our primary objective is to develop a method, in term of group actions and orbit-stabilizers, to 
understand the geometry of the isomorphism classes of finite-dimensional $\upomega$-Lie algebras over $\K$.

In the past twenty years, algebraic structures and classifications of finite-dimensional $\upomega$-Lie algebras over a field of characteristic zero have been studied extensively; see for example, \cite{CNY23, CZ17, CRSZ26, Oub24, Oub25} and \cite{Zha21}. Among these studies, low-dimensional $\upomega$-Lie algebras play an irreplaceable role in understanding  representations, extensions, derivations, and automorphisms of $\upomega$-Lie algebras.
For example, it was proved in \cite[Section 7]{CZ17} that all non-Lie finite-dimensional simple complex $\upomega$-Lie algebras must be $3$-dimensional.  Note that $3$-dimensional real $\upomega$-Lie algebras were classified by
\cite[Theorem 2.1]{Nur07} as well as the first but relatively rough classification of all $3$-dimensional complex 
$\upomega$-Lie algebras appeared in \cite[Theorem 2]{CLZ14}. Recently, the study of related $\upomega$-algebra structures have also became popular in the area of nonassociative algebras; see \cite{CW23, SHZC25, Zho25} and \cite{CNY26}. 
One goal of this article is to explore the geometric structure of the affine variety of all 3-dimensional non-Lie $\upomega$-Lie 
algebras over $\K$, and particularly, to give a complete classification of these algebras up to $\upomega$-algebra isomorphism, when $\K$ is algebraically closed.

Note that an $\upomega$-Lie algebra consists of three components: the underlying vector space $L$, the bracket product $[-,-]$, and the bilinear form $\upomega$. Thus when $L$ is fixed, the algebraic structure of an $\upomega$-Lie algebra will be determined by  $[-,-]$ and $\upomega$. In the low-dimensional classifications (when $\K=\C$ and $\R$) presented in \cite{CLZ14, CZ17}, our approach consistently involved first determining the bracket product, followed by the specification of the bilinear form $\upomega$. In contrast to our previous methods, the present article carefully examine  relationships between the bracket  $[-,-]$ and the form $\upomega$, taking a strategy that begins with the determination of $\upomega$. Moreover, we use the language of group actions and orbit-stabilizers to understand the geometry of the isomorphism classes of $\upomega$-Lie algebras of a fixed dimension over $\K$.
The following theorem is the first main result in our article. 

\begin{thm}\label{thm1}
Let $\upomega$ be a nonzero skew-symmetric bilinear form on an $n$-dimensional vector space $V$ over a field $\K$ of characteristic $\neq 2$ and $\LL_\upomega(\K)$ be the affine variety of all $\upomega$-Lie algebras on $V$ with respect to $\upomega$. 
Suppose $G_\upomega$ denotes the stabilizer of $\upomega$ in $\GL(V)$. Then there exists a one-to-one 
correspondence between the set of isomorphism classes of $\LL_\upomega(\K)$ and the orbit space of 
$\LL_\upomega(\K)$ under an action of $G_\upomega$.
\end{thm}

The geometric structure of the variety of finite-dimensional Lie algebras has been a classical topic in Lie theory and differential geometry, closely related to the theory of  contractions, deformations, and cohomology of Lie algebras; see for example, \cite[Chapter 5]{GK96}. As the first step of exploring the geometric structure of finite-dimensional $\upomega$-Lie algebras, we are interested in the affine variety  $\A_3(\K)$ of all 3-dimensional non-Lie $\upomega$-Lie algebras over $\K$.  The second purpose of this article is to use techniques from computational ideal theory to explore the geometric structure of $\A_3(\K)$.  Our method and its analogue have been successfully applied to the study of other nonassociative algebras (see  \cite{CCZ21, CZ23a, CRSZ25} and \cite{CQRZ26}) and have inspired numerous subsequent research; see \cite{DW24, Che25, DS25}, and \cite{CZ26}.
We state the second main result as follows. 

\begin{thm}\label{thm2}
Let $\upomega$ be a nonzero skew-symmetric bilinear form on a $3$-dimensional vector space $V$ over a field $\K$ of characteristic $\neq 2$ and $\A_3(\K)$ be the variety of all $\upomega$-Lie algebras on $V$. Then $\A_3(\K)$ is a $6$-dimensional irreducible affine variety and the corresponding vanishing ideal is a complete intersection. 
\end{thm}

\noindent More interestingly, we observe that this variety $\A_3(\K)$ is both a geometric complete intersection and an algebraic complete intersection; see Remark \ref{rem3.8}.

As an application of Theorems \ref{thm1} and \ref{thm2}, the third main result of this article completely classifies 
all 3-dimensional non-Lie $\upomega$-Lie algebras over an algebraically closed field of characteristic $\neq 2$, up to $\upomega$-Lie algebra isomorphism. 

\begin{thm}\label{thm3}
Let $\K$ be an algebraically closed field of characteristic $\neq 2$ and $L\in \A_3(\K)$ be a $3$-dimensional non-Lie $\upomega$-Lie algebra over $\K$. Then
\begin{enumerate}
  \item $L$ must be isomorphic to one of $\{A, B, C_\upalpha,D\mid \upalpha\in\K^\times\}$ that appears in Section {\rm \ref{sec4};}
  \item All these $\upomega$-Lie algebras {\rm (}\ref{A}{\rm)},{\rm (}\ref{B}{\rm)},{\rm (}\ref{C}{\rm)} and  {\rm (}\ref{D}{\rm)}  are mutually non-isomorphic.
\end{enumerate}
\end{thm}
\noindent Note that our previous work \cite{CLZ14} used linear algebra and representation theory methods to give a rough classification of
3-dimensional non-Lie $\upomega$-Lie algebras over $\C$, but we did not examine when two such algebras are isomorphic of $\upomega$-Lie algebras. Here we will use a different method from group actions and orbit-stabilizers to prove Theorem \ref{thm3}, which can also  be viewed as a generalization and an enhancement of \cite[Theorem 2]{CLZ14}. 

We organize this article as follows. In Section \ref{sec2}, we work on a general $n$-dimensional $\upomega$-Lie algebras $(n\geqslant 3)$ over a field $\K$ of characteristic $\neq 2$ and explore relationships between bracket products and bilinear forms in $\upomega$-Lie algebras, showing that two isomorphic $\upomega$-Lie algebras have equivalent bilinear forms; see Proposition \ref{prop2.4}. This result motivates us to use the language of group actions and orbit-stabilizers to develop a procedure for classifying all $n$-dimensional non-Lie $\upomega$-Lie algebras over $\K$. This section also contains a proof of Theorem \ref{thm1}.  Section \ref{sec3}
explores the geometric structure of the affine variety of $\A_3(\K)$, and in particular provides a proof of Theorem \ref{thm2}. In this section, all $\upomega$-Lie algebras are 3-dimensional and non-Lie over $\K$. Moreover, our proofs heavily rely on the theory of Gr\"{o}bner bases, and some notations, such as $\LM(\ell)$ and $\LC(\ell)$ denoting the leading monomial and the leading coefficient of a polynomial $\ell$, respectively, are standard in computational ideal theory; see for example, \cite[Chapter 1]{DK15}. In Section \ref{sec4}, we apply Theorems \ref{thm1} and \ref{thm2} to obtain a complete classification of all 
3-dimensional non-Lie $\upomega$-Lie algebras over an algebraically close field of characteristic $\neq 2$. This section consists of four subsections: the first three  are devoted to finding all such $\upomega$-Lie algebras, while the last  completes the proof of Theorem \ref{thm3}.  Section \ref{sec5} contains two remarks illustrating that our method can be used to construct new $\upomega$-Lie algebras over any fields of characteristic $\neq 2$. This also suggests that our approach might have potential applications in classifying finite-dimensional $\upomega$-Lie algebras over the rational field or finite fields.

Throughout this article, $\C$ and $\R$ denote the field of complex and real numbers, respectively.  We always assume that $\K$ denotes an arbitrary field of characteristic $\neq 2$, except in Section \ref{sec4}, where $\K$ is required to be algebraically closed. All vector spaces and algebras are finite-dimensional and over $\K$. The symbol $\ideal{\ast}$ denotes the ideal generated by $\ast$ or the subspace spanned by $\ast$.  We also write ``grevlex'' for the graded reverse lexicographical monomial ordering on a polynomial ring over $\K$ and write $\K^\times$ for the set of nonzero elements of $\K$.

\section{Varieties of $\upomega$-Lie Algebras}\label{sec2}
\setcounter{equation}{0}
\renewcommand{\theequation}
{2.\arabic{equation}}
\setcounter{theorem}{0}
\renewcommand{\thetheorem}
{2.\arabic{theorem}}

\noindent This section first explores some fundamental properties and relations between bracket products and bilinear forms in $\upomega$-Lie algebras. We also provide a method, using the viewpoint of group action and orbit-stabilizer  theory, to classify all $\upomega$-Lie algebras of a fixed dimension, up to isomorphism of $\upomega$-Lie algebras.

Note that there are no non-Lie $\upomega$-Lie algebras in dimensions 1 or 2 (see \cite[Section 1]{Nur07}). Thus we assume that $\dim_\K(L)\geqslant 3$ in this section. Let us begin with the following simple fact. 

\begin{prop}\label{prop2.1}
If $(L,[-,-],\upomega)$ is an $\upomega$-Lie algebra over $\K$, then $\upomega$ is skew-symmetric. 
\end{prop}

\begin{proof}
For any element $y\in L$, since $(L,[-,-])$ is skew-symmetric and $\cha(\K)\neq 2$, it follows that $[y,y]=0$.
Setting $z=y$ in (\ref{Jacobi}) obtains that $(\upomega(x,y)+\upomega(y,x))y+\upomega(y,y)x=0$ for all $x\in L$. 
Since $\dim_k(L)\geqslant 3$, we may take an element $x$ such that $\{x,y\}$ are linearly independent over $\K$. Thus
$\upomega(y,y)=0$. Hence, $\upomega(y_1+y_2,y_1+y_2)=0$ for all $y_1,y_2\in L$. Expanding this equation we see that 
$\upomega(y_1,y_2)=-\upomega(y_2,y_1)$, showing that $\upomega$ is skew-symmetric.
\end{proof}

\begin{prop} 
If $\upomega_1$ and $\upomega_2$ are two bilinear forms on $(L,[-,-])$ such that both $(L,[-,-],\upomega_1)$ and $(L,[-,-],\upomega_2)$ are $\upomega$-Lie algebras, then $\upomega_1=\upomega_2$.
\end{prop}

\begin{proof}
Choose any three elements $x,y,z\in L$ such that $z$ doesn't lie in the linear subspace spanned by $\{x,y\}$ and consider the corresponding (\ref{Jacobi}) in $(L,[-,-],\upomega_1)$ and $(L,[-,-],\upomega_2)$:
\begin{eqnarray*}
~[[x,y],z]+[[y,z],x]+[[z,x],y]&=&\upomega_1(x,y)z+\upomega_1(y,z)x+\upomega_1(z,x)y \\
~[[x,y],z]+[[y,z],x]+[[z,x],y]&=&\upomega_2(x,y)z+\upomega_2(y,z)x+\upomega_2(z,x)y.
\end{eqnarray*}
Combining these two equations obtains
$$(\upomega_2(x,y)-\upomega_1(x,y))z=(\upomega_1(y,z)-\upomega_2(y,z))x+(\upomega_1(z,x)-\upomega_2(z,x))y.$$
Since $z$ can not be linearly expressed by $x$ and $y$, it follows that $\upomega_1(x,y)=\upomega_2(x,y)$. Therefore, $\upomega_1=\upomega_2$.
\end{proof}

\begin{rem}{\rm
This result means that if the bracket product $[-,-]$ on a vector space $L$ is fixed, then there exists at most one bilinear form $\upomega$ to make the triple $(L,[-,-],\upomega)$ becomes an $\upomega$-Lie algebra. However, we will see that if we fix a bilinear form $\upomega$ on a vector space $L$, there might be many bracket products $[-,-]$ such that $(L,[-,-],\upomega)$ are  $\upomega$-Lie algebras.
\hbo}\end{rem}

Recall that two skew-symmetric bilinear forms $\upomega_1$ and $\upomega_2$ on a finite-dimensional vector space $W$ over $\K$ are \textit{equivalent} if  there exists an invertible linear map $g:W\ra W$ such that 
\begin{equation}
\label{ }
\upomega_1(x,y)=\upomega_2(g(x),g(y))
\end{equation}
for all $x,y\in W$. This means that $\upomega_1$ and $\upomega_2$ can be transformed into each other by basis change in $W$. In terms of matrix language, if we fix a basis of $W$ and use $A_i$ to denote the resulting skew-symmetric matrices of $\upomega_i$, then 
\begin{equation}
\label{ }
A_2=Q^t\cdot A_1\cdot Q
\end{equation}
for some invertible matrix $Q$, where $Q^t$ denotes the transpose of $Q$. In other words,  $\upomega_1$ and $\upomega_2$ are equivalent if and only if their matrices $A_1$ and $A_2$ are congruent.

We say that two $\upomega$-Lie algebras $(L_1,[-,-]_1,\upomega_1)$ and $(L_2,[-,-]_2,\upomega_2)$ are \textit{isomorphic} if there exists an algebraic isomorphism $g:L_1\ra L_2$ such that $\upomega_1(x,y)=\upomega_2(g(x),g(y))$ for all $x,y\in L_1$; compared with \cite[Definition 1.1]{CZZZ18}. Clearly, we have

\begin{prop} \label{prop2.4}
Suppose that $(L_1,[-,-]_1,\upomega_1)$ and $(L_2,[-,-]_2,\upomega_2)$ are two isomorphic $\upomega$-Lie algebras over $\K$. Then {\rm (1)} $\dim_\K(L_1)=\dim_\K(L_2)${\rm; (2)}  $\upomega_1$ and $\upomega_2$ are equivalent.
\end{prop}

This immediate result, together with Proposition \ref{prop2.1}, allows us to provide the following  procedure to understand the variety $\A_n(\K)$ of all $n$-dimensional $\upomega$-Lie algebras over $\K$ and classify $\upomega$-Lie algebras in $\A_n(\K)$:

\begin{enumerate}
  \item We may first classify all $n \times n$ skew-symmetric matrices over $\K$ up to congruence, which is equivalent to classifying all skew-symmetric bilinear forms of size $n$ over $\K$ up to equivalence. This is a well-established topic in linear algebra.
  \item  Secondly, we need to determine all congruence canonical forms of these skew-symmetric matrices. Note that distinct canonical forms correspond to non-isomorphic $\upomega$-Lie algebras. Thus, if two $\upomega$-Lie algebras over $\K$ correspond to different congruence canonical forms, they cannot be isomorphic over $\K$. 
  \item For each congruence canonical form, we write $\upomega$ for the corresponding bilinear form 
  and  denote by $\LL_\upomega(\K)$ the subvariety of $\A_n(\K)$, consisting of all non-Lie $\upomega$-Lie algebra structures with respect to $\upomega$. We need to analyze the geometric structure of $\LL_\upomega(\K)$.
  \item Fix an $n$-dimensional vector space $V$ over $\K$ and a skew-symmetric bilinear form $\upomega$. To
  distinguish two elements $(V,[-,-],\upomega)$ and $(V,[-,-]',\upomega)$  in  $\LL_\upomega(\K)$, we denote by
  $G_\upomega$ the stabilizer of $\upomega$ in $\GL(V)$, i.e.,
  \begin{equation}
\label{ }
G_\upomega:=\{g\in\GL(V)\mid \upomega(g(v),g(w))=\upomega(v,w),\textrm{ for all }v,w\in V\}.
\end{equation}
We will see below that the group  $G_\upomega$ can act on $\LL_\upomega(\K)$, and it is very helpful in classifying 
$\upomega$-Lie algebra structures in  $\LL_\upomega(\K)$.
\end{enumerate}

To realize the first step of the above procedure, we present the following well-known result (Theorem \ref{thm2.5}) on the congruence classification of $n \times n$ skew-symmetric matrices over $\K$. For the reader's convenience, we include a proof of this result, because we were unable to find a suitable reference that provides a detailed proof.

We define
\begin{equation}
\label{eq2.4}
J:=\begin{pmatrix}
  0    & 1   \\
   -1   & 0 
\end{pmatrix}.
\end{equation}

\begin{thm} \label{thm2.5}
Let $A$ be an $n \times n$ skew-symmetric matrix of rank $m$ over $\K$. Then $A$ is congruent to
\begin{equation}
\label{eq2.5}
J_m:=\dia\{J,\dots,J,0,\dots,0\}
\end{equation}
where the number of $J$ is $\frac{m}{2}$ and the number of zero is $n-m$.
\end{thm}

\begin{proof}
We may assume that $A\neq 0$. Note that all diagonals in $A$ are zero (because $A^t=-A$ and $\cha(\K)\neq 2$), thus $n\geqslant 2$. We use induction on $n$ to prove this result. 

Consider the case $n=2$. We may assume that $A=\begin{pmatrix}
      0&a    \\
      -a&0  
\end{pmatrix}$ for some $a\in\K^\times$. Define $Q=\dia\left\{1,a^{-1}\right\}$. Then
$Q^t\cdot A\cdot Q=\begin{pmatrix}
   1   & 0   \\
    0  & a^{-1} 
\end{pmatrix}\begin{pmatrix}
      0&a    \\
      -a&0  
\end{pmatrix} \begin{pmatrix}
   1   & 0   \\
    0  & a^{-1} 
\end{pmatrix}=\begin{pmatrix}
  0    & 1   \\
   -1   & 0 
\end{pmatrix}=J.$ Thus $A$ is congruent to $J$.

Suppose that $n>2$.  Since $A\neq 0$, the base change allows us to write 
$$A=\begin{pmatrix}
    a\cdot J  & B   \\
    -B^t  & D 
\end{pmatrix}$$
for some $a\in\K^\times$ and a skew-symmetric matrix $D$ of size $n-2$. Define $Q=\begin{pmatrix}
   I_2   & a^{-1}\cdot J\cdot B   \\
0  &  I_{n-2}
\end{pmatrix}$. Since $J^t=-J$, we see that
$Q^t=\begin{pmatrix}
   I_2   & 0   \\
   -a^{-1}\cdot B^t\cdot J   &  I_{n-2}
\end{pmatrix}.$
Note that $J^2=-I_2$. Thus
\begin{eqnarray*}
Q^t\cdot A\cdot Q&=&\begin{pmatrix}
   I_2   & 0   \\
   -a^{-1}\cdot B^t\cdot J   &  I_{n-2}
\end{pmatrix} \begin{pmatrix}
    a\cdot J  & B   \\
    -B^t  & D 
\end{pmatrix} \begin{pmatrix}
   I_2   & a^{-1}\cdot J\cdot B   \\
0  &  I_{n-2}
\end{pmatrix}\\
&=&\begin{pmatrix}
      a\cdot J  & B     \\
     0 & D- a^{-1}\cdot B^t\cdot J\cdot B
\end{pmatrix}\begin{pmatrix}
   I_2   & a^{-1}\cdot J\cdot B   \\
0  &  I_{n-2}
\end{pmatrix}\\
&=&\begin{pmatrix}
      a\cdot J  & 0     \\
     0 & D- a^{-1}\cdot B^t\cdot J\cdot B
\end{pmatrix}.
\end{eqnarray*}
This means that $A$ is congruent to $\dia\{a\cdot J, D- a^{-1}\cdot B^t\cdot J\cdot B\}$. Note that 
$D- a^{-1}\cdot B^t\cdot J\cdot B$ is a skew-symmetric matrix of size $n-2$ and its rank equals  $m-2$ (because of $a\neq 0$). The induction hypothesis shows that it is congruent to $J_{m-2}$. Therefore, $A$ is congruent to $J_m$.
\end{proof}

The rest of this section is  devoted to giving a proof of Theorem \ref{thm1}.

Suppose that $\dim_\K(V)=n$ and fix a skew-symmetric bilinear form $\upomega$ of rank $m$. 
Choose a suitable basis for $V$ and by Theorem \ref{thm2.5}, we see that the skew-symmetric matrix corresponding to $\upomega$
is congruent to $J_m$ appeared in (\ref{eq2.5}). It is well-known in the theory of classical groups that
\begin{equation}
\label{ }
G_\upomega=\{g\in\GL_n(\K)\mid g^t\cdot J_m \cdot g=J_m\}
\end{equation}
is the singular symplectic group of rank $m$ over $\K$; see for example, \cite[Example 1.3.15]{Gec13} or \cite[Section 3.3]{Wan02}.
Moreover, the group $G_\upomega$ also can be viewed as the glueing group of the symplectic group of degree $m$ and the
general linear group of degree $n-m$ via the full bimodule; see \cite[Section 2 and Example 6.4]{CSW21}.

Given an element $g\in G_\upomega$ and
a bracket $[-,-]\in \LL_\upomega(\K)$,  we define a new bracket 
\begin{equation}
\label{action}
[x,y]_g:=g[g^{-1}(x),g^{-1}(y)]
\end{equation}
where  $x,y\in V$. Then

\begin{lem}
The new bracket  $[-,-]_g$ is an $\upomega$-Lie algebra structure on $V$. 
\end{lem}

\begin{proof}
Suppose that $x,y,z\in V$ are arbitrary elements. It follows that 
 $[y,x]_g=g[g^{-1}(y),g^{-1}(x)]=-g[g^{-1}(x),g^{-1}(y)]=-[x,y]_g$. Thus
  $[-,-]_g$ is skew-symmetric. To show that the $\upomega$-Jacobi identity holds for   $[-,-]_g$, we observe that
  $$[[x,y]_g,z]_g=g[g^{-1}([x,y]_g),g^{-1}(z)]=g[[g^{-1}(x),g^{-1}(y)],g^{-1}(z)].$$
  Thus
  \begin{eqnarray*}
&&[[x,y]_g,z]_g+[[y,z]_g,x]_g+[[z,x]_g,y]_g \\
& = & g\bigg([[g^{-1}(x),g^{-1}(y)],g^{-1}(z)]+[[g^{-1}(y),g^{-1}(z)],g^{-1}(x)]+[[g^{-1}(z),g^{-1}(x)],g^{-1}(y)]\bigg) \\
 & =& g\bigg(\upomega(g^{-1}(x),g^{-1}(y))\cdot g^{-1}(z)+\upomega(g^{-1}(y),g^{-1}(z))\cdot g^{-1}(x)+\upomega(g^{-1}(z),g^{-1}(x))\cdot g^{-1}(y)\bigg) \\
 &=& \upomega(g^{-1}(x),g^{-1}(y))z+\upomega(g^{-1}(y),g^{-1}(z))x+\upomega(g^{-1}(z),g^{-1}(x))y\\
 &=& \upomega(x,y)z+\upomega(y,z)x+\upomega(z,x)y.
\end{eqnarray*}
This implies that the new bracket  $[-,-]_g$ defines an $\upomega$-Lie algebra structure on $V$. 
\end{proof}

Moreover, $G_\upomega$ can act on the affine variety $\LL_\upomega(\K)$ defined by
\begin{equation}
\label{eq2.8}
\left(g\cdot [-,-]\right)(x,y):=[x,y]_g=g[g^{-1}(x),g^{-1}(y)]
\end{equation}
where $g\in G_\upomega, [-,-]\in L_\upomega(\K)$, and $x,y\in V$. In fact, it it clear that $[x,y]_1=[x,y]$, which means that the identify element  of $G_\upomega$ fixes the original bracket $[-,-]$. For two elements $g,h\in G_\upomega$, note that 
\begin{eqnarray*}
((gh)\cdot [-,-])(x,y)&=&[x,y]_{gh}=(gh)([(gh)^{-1}(x),(gh)^{-1}(y)])\\
&=&g\Big(h([h^{-1}(g^{-1}(x)),h^{-1}(g^{-1}(y))])\Big)=g([g^{-1}(x),
g^{-1}(y)]_h)\\
&=&(g\cdot [-,-]_h)(x,y)=(g\cdot (h\cdot [-,-]))(x,y).
\end{eqnarray*}
Hence, $(gh)\cdot [-,-]=g\cdot (h\cdot [-,-])$, as we desire.

We are ready to prove Theorem \ref{thm1}. 

\begin{proof}[Proof of Theorem \ref{thm1}]
Clearly,  it suffices to show that two $\upomega$-Lie algebras in $\LL_\upomega(\K)$ are isomorphic if 
and only if they are in the same orbit, under the action of $G_\upomega$ defined in (\ref{eq2.8}). 

Suppose that $[-,-]$ and $[-,-]'$ are two $\upomega$-Lie algebras in $L_\upomega(\K)$. 
If they are isomorphic, then there exists an element $g\in G_\upomega$ such that
$g[x,y]'=[g(x),g(y)]$ for all $x,y\in V$. Thus
$$[x,y]'=g^{-1}[g(x),g(y)]=[x,y]_{g^{-1}}=\left(g^{-1}\cdot [-,-]\right)(x,y)$$
which means that  $[-,-]$ and $[-,-]'$ are in the same orbit. Conversely, 
suppose that $[-,-]'=g\cdot [-,-]$ for some $g\in G_\upomega$. Then 
$[x,y]'=\left(g\cdot [-,-]\right)(x,y)=[x,y]_g=g[g^{-1}(x),g^{-1}(y)]$. Thus
$$g^{-1}[x,y]'=[g^{-1}(x),g^{-1}(y)]$$
for all $x,y\in V$. Namely, $g^{-1}$ is an algebra isomorphism.  Since $g^{-1}\in G_\upomega$ is also $\upomega$-preserving, we see that $[-,-]'$ and $[-,-]$ are isomorphic as $\upomega$-Lie algebras. 
\end{proof}

\section{Geometric Structure of $\A_3(\K)$} \label{sec3}
\setcounter{equation}{0}
\renewcommand{\theequation}
{3.\arabic{equation}}
\setcounter{theorem}{0}
\renewcommand{\thetheorem}
{3.\arabic{theorem}}

\noindent In this section, we explore the geometric structure of the affine variety $\A_3(\K)$ and prove Theorem \ref{thm2}.
Suppose that $V$ denotes a 3-dimensional vector space over $\K$, spanned by a basis $\{x,y,z\}$. 

To understand the affine variety $\A_3(\K)$ of all $\upomega$-Lie algebra structures on $V$, we assume that $(L,[-,-],\upomega)$ denotes a generic element in $\A_3(\K)$.  By Theorem \ref{thm2.5}, we see that there exists a unique nonzero congruence canonical form for all $3\times 3$ nonzero skew-symmetric matrices over $\K$:
$$J_2=\begin{pmatrix}
     0 &1&0    \\
      -1&0&0    \\ 
        0 &0&0    \\ 
\end{pmatrix}.$$
Hence, we may take the $\upomega$ in $L$ as the skew-symmetric bilinear form defined by 
\begin{equation*}
\label{ }
\upomega(x,y)=1\textrm{ and } \upomega(x,z)=\upomega(y,z)=0.
\end{equation*}

Hence, $\A_3(\K)=\LL_{\upomega}(\K)$.  We may assume that 
\begin{equation}\label{eq3.1}
\begin{aligned}
[x,y] & =~  a_1x+a_2y+a_3z \\
[x,z] & =~  b_1x+b_2y+b_3z \\
[y,z] & =~  c_1x+c_2y+c_3z
\end{aligned}
\end{equation}
where all $a_i,b_i,c_i$ are scalars in $\K$. The skew-symmetry of the bracket product in $L$ shows that
the $\upomega$-algebra structure of $L$ could be determined by these scalars $\{a_i,b_i,c_i\mid i=1,2,3\}$. 
Moreover, the (\ref{Jacobi}) allows us to view $\A_3(\K)$ as an affine variety within the 9-dimensional affine space $\K^9$.

Let $R:=\K[x_i,y_i,z_i\mid i=1,2,3]$ denote the polynomial ring in $9$ variables over $\K$, where $x_i, y_i, z_i$ 
are dual to $a_i,b_i,c_i$, respectively. To understand the geometric structure of $\A_3(\K)$, we need to determine
the vanishing ideal $I(\A_3(\K))$ in $R$ and the coordinate ring $R/I(\A_3(\K))$. 

Define three polynomials in $R$ as follows:
\begin{equation}\label{eq3.2}
\begin{aligned}
f_1 & :=  x_2z_1 + y_3z_1 - x_1z_2 - y_1z_3 \\
f_2 & :=  x_3y_1 - x_1y_3 + x_3z_2 - x_2z_3 + 1\\
f_3 & :=   x_2y_1 - x_1y_2 - y_3z_2 + y_2z_3
\end{aligned}
\end{equation}
and write $\p:=\ideal{f_1,f_2,f_3}$ for the ideal generated by these three polynomials in $R$. 

\begin{lem}\label{lem3.1}
$\p\subseteq I(\A_3(\K))$.
\end{lem}

\begin{proof}
It suffices to show that $f_1,f_2,f_3$ all vanish on the generic $\upomega$-Lie algebra $L$. Computing   
(\ref{Jacobi}) with the variables running over any three of $\{x,y,z\}$ (not necessarily non-repeated), we may obtain 
the following six relations among $\{a_i,b_i,c_i\mid i=1,2,3\}$:
\begin{eqnarray*}
a_2b_1 - a_1b_2 - b_3c_2 + b_2c_3=0 && -a_2b_1 + a_1b_2 + b_3c_2 - b_2c_3=0\\
a_2c_1 + b_3c_1 - a_1c_2 - b_1c_3=0 &&-a_2c_1 - b_3c_1 + a_1c_2 + b_1c_3=0\\
a_3b_1 - a_1b_3 + a_3c_2 - a_2c_3 + 1=0&& -a_3b_1 + a_1b_3 - a_3c_2 + a_2c_3 - 1=0.
\end{eqnarray*}
Clearly, the three relations on the right-hand column are redundant. The first three relations on the left-hand column show that $f_i(L)=0$ for $i=1,2,3$. Therefore, $\p\subseteq I(\A_3(\K))$.
\end{proof}

\begin{prop}\label{prop3.2}
$\A_3(\K)=V(\p)$.
\end{prop}

\begin{proof}
Since $\A_3(\K)$ is contained in its closure $\overline{\A_3(\K)}$ and the latter equals $V(I(\A_3(\K)))$, it follows from
Lemma \ref{lem3.1} that $\A_3(\K)\subseteq \overline{\A_3(\K)}=V(I(\A_3(\K)))\subseteq V(\p)$. Conversely,
for each point $L:=(a_1,a_2,a_3,b_1,b_2,b_3,c_1,c_2,c_3)$ in $V(\p)$, we see that
$f_i(L)=0$ for $i=1,2,3$. A direct calculation on (\ref{Jacobi}) shows that the point $L$, together with the three relations on the left-hand column appeared in the proof of Lemma \ref{lem3.1}, gives an 
$\upomega$-Lie algebra structure on $V$. This means that $V(\p)\subseteq\A_3(\K)$. Hence, $\A_3(\K)=V(\p)$.
\end{proof}

To show that $\p$ is a prime ideal in $R$, we need to define
\begin{equation}
\label{ }
M:=\begin{pmatrix}
   x_2+y_3   &  -x_1&-y_1  \\
    0  & x_3&-x_2\\
    0&-y_3&y_2 
\end{pmatrix}.
\end{equation}
Clearly, $\det(M)=(x_2+y_3)(x_3y_2-x_2y_3)\neq 0$ is not a zero divisor in $R$ because $R$ is an integral domain.
Throughout this section, we choose the grevlex ordering on $R$ with $x_1>x_2>x_3>y_1>y_2>y_3>z_1>z_2>z_3$ and define 
\begin{equation}
\begin{aligned}
g & :=  x_1y_2z_1 + y_1y_3z_1 - x_1y_1z_2 + y_3z_1z_2 - y_1^2z_3 - y_2z_1z_3\\
h & :=  x_1x_3y_2 - x_1x_2y_3 + x_2x_3z_2 + x_3y_3z_2 - x_2^2z_3 - x_3y_2z_3 + x_2.
\end{aligned}
\end{equation}

\begin{lem} \label{lem3.3}
$\B:=\{f_1,f_2,f_3,g,h\}$ is a Gr\"{o}bner basis for $\p$.
\end{lem}

\begin{proof} 
To prove this statement, we will use the Buchberger’s algorithm; see for example, \cite[Algorithm1.1.9]{DK15}. 
Note that $\p=\ideal{f_1,f_2,f_3}$. Since the leading monomials of $f_1$ and $f_2$ are coprime, thus $\spol(f_1,f_2)$ reduces to zero modulo  $\{f_1,f_2,f_3\}$ by \cite[Proposition 4, page 106]{CLO15}. Moreover, let us recall the definition of $s$-polynomial of any two polynomials $f$ and $g$:
$$\spol(f,g)=\frac{\LC(g)\cdot t}{\LM(f)}\cdot f- \frac{\LC(f)\cdot t}{\LM(g)}\cdot g$$
where $t:=\lcm(\LM(f),\LM(g))$ denotes the least common multiple. Thus
$$\spol(f_1,f_3)=\frac{1\cdot x_2y_1z_1}{x_2z_1}\cdot f_1- \frac{1\cdot x_2y_1z_1}{x_2y_1}\cdot f_3=y_1f_1-z_1f_3=g.$$
Clearly, no term of $g$ is divisible by anyone of $\{\LM(f_1),\LM(f_2),\LM(f_3)\}$. Thus the normal form of $g$ with respect to $\{f_1,f_2,f_3\}$ is itself. Similarly, since 
$$\spol(f_2,f_3)=x_2f_2-x_3f_3=h$$
whose normal form with respect to $\{f_1,f_2,f_3\}$ being itself as well, we can include $g$ and $h$ into 
$\{f_1,f_2,f_3\}$ and obtain the set $\B$. 

To complete the proof, we need to check that any two elements of $\B$
wouldn't produce an $s$-polynomial with nonzero normal form with respect to $\B$. This is immediate for any two in $\{f_1,f_2,f_3\}$ and any two in $\B$ whose leading monomials are coprime. Thus it suffices to show that 
the normal form of $\spol(f_1,g)$ is zero with respect to $\B$. Applying \cite[Algorithm1.1.6]{DK15}, a direct verification shows that the normal form of $\spol(f_1,g)$ is zero with respect to $\B$. In fact,
$$\spol(f_1,g)=(y_2z_3-y_3z_2)f_1-(y_3z_1 - x_1z_2 - y_1z_3)f_3.$$
Therefore, it follows from \cite[Theorem 3, page 105]{CLO15} that $\B$ is a Gr\"{o}bner basis for $\p$.
\end{proof}

\begin{lem} \label{lem3.4}
$\{f_1,f_2\}$ is a Gr\"{o}bner basis for the ideal $\ideal{f_1,f_2}$.
\end{lem}

\begin{proof}
Note that $\LM(f_1)=x_2z_1$ and $\LM(f_2)=x_3y_1$. As they are coprime, it follows from \cite[Proposition 4, page 106]{CLO15} that the $s$-polynomial $\spol(f_1,f_2)$ reduces to zero modulo  $\{f_1,f_2\}$. By \cite[Theorem 3, page 105]{CLO15}, we deduce that $\{f_1,f_2\}$ is a Gr\"{o}bner basis for  $\ideal{f_1,f_2}$.
\end{proof}

\begin{lem}   \label{lem3.5}
$\ideal{f_1}$ and $\ideal{f_1,f_2}$ are prime ideals in $R$. 
\end{lem}

\begin{proof}
Note that $f_1$ is irreducible in $R$ and $R$ is factorial, thus $f_1$ is a prime element. This implies that $\ideal{f_1}$ is a prime ideal. To show that $\ideal{f_1,f_2}$ is prime, we write $R_2$ for $R/\ideal{f_1,f_2}$ and it suffices to show that
it is an integral domain. Denote by $\Delta$ the determinant of $\begin{pmatrix}
    x_1  & y_1   \\
    -x_3  & x_2 
\end{pmatrix}$, i.e., $\Delta:=x_1x_2+x_3y_1\neq 0$. 

We \textit{claim} that the class of $\Delta$ is not a zero divisor in $R_2$. If this claim holds, then $R_2$ can be embedded in the localization $R_2\left[\Delta^{-1}\right]$. Moreover, note that $\Delta\neq 0$ in $R$. Solving the following linear system
$$\begin{pmatrix}
    x_1  & y_1   \\
    -x_3  & x_2 
\end{pmatrix}\begin{pmatrix}
      z_2    \\
      z_3  
\end{pmatrix}=\begin{pmatrix}
      x_2z_1 + y_3z_1    \\
      x_3y_1 - x_1y_3-1  
\end{pmatrix}$$
implies that 
$$R_2\left[\Delta^{-1}\right]\cong \K[x_1,x_2,x_3,y_1,y_2,y_3,z_1]\left[\Delta^{-1}\right]$$
and the latter is an integral domain  because it is a localization of a polynomial ring. Hence, $R_2$, as a subring of 
$R_2\left[\Delta^{-1}\right]$ is also an integral domain.

Hence, it suffices to prove the claim above. Apparently, this claim is equivalent to the statement that
the colon ideal $\ideal{f_1,f_2}: \Delta$ is equal to $\ideal{f_1,f_2}$. By \cite[Section 1.2.4]{DK15}, we see that
\begin{equation}
\label{ }
\ideal{f_1,f_2}: \Delta=\Delta^{-1}\cdot \left(\ideal{f_1,f_2}\cap \ideal{\Delta}\right).
\end{equation} 
A standard method of finding Gr\"{o}bner basis of the intersection of two ideals can be applied. Let us introduce a new variable $z$ and consider the ideal $J:=\ideal{f_1,f_2}\cdot z+\ideal{\Delta}\cdot (1-z)$ in $R[z]$.
Extending the grevlex ordering to $R[z]$ with $z>x_1>x_2>\cdots>z_3$ and using the same method in Lemma \ref{lem3.3}, we obtain a Gr\"{o}bner basis $\D=\{\Delta\cdot f_1, \Delta\cdot f_2,z\cdot f_1,z\cdot f_2, h_1,h_2\}$ for $J$, where
\begin{eqnarray*}
h_1 &:= & zx_1x_2 + zx_1y_3 - zx_3z_2 + zx_2z_3 - x_1x_2 - x_3y_1 - z \\
h_2 &:= & zx_1^2z_2 - zx_3z_1z_2 + zx_1y_1z_3 - zy_3z_1z_3 + zx_1z_2z_3 + zy_1z_3^2 - 
        x_1x_2z_1 - x_3y_1z_1 - zz_1. 
\end{eqnarray*}
According to \cite[Algorithm1.2.1]{DK15}, to obtain a Gr\"{o}bner basis for $\ideal{f_1,f_2}\cap \ideal{\Delta}$,
 we need to remove all elements of $\D$ that involve the extra variable $z$. Clearly,
 $$\{\Delta\cdot f_1, \Delta\cdot f_2\}$$
is a Gr\"{o}bner basis for $\ideal{f_1,f_2}\cap \ideal{\Delta}$, and thus $\{f_1, f_2\}$
is a Gr\"{o}bner basis for $\ideal{f_1,f_2}: \Delta$. 
Therefore, $\ideal{f_1,f_2}: \Delta=\ideal{f_1,f_2}$. This completes the proof of the claim. 
\end{proof}

\begin{lem} \label{lem3.6}
The class of $\det(M)$ in $R/\p$ is not a zero divisor.
\end{lem}

\begin{proof}
We use the same method for the claim appeared in Lemma \ref{lem3.5} to prove this statement.  
It suffices to verify that the colon ideal $\p: \det(M)$
is equal to $\p$. By \cite[Section 1.2.4]{DK15}, we see that
\begin{equation}
\label{ }
\p: \det(M)=\frac{1}{\det(M)}\cdot (\p\cap \ideal{\det(M)}).
\end{equation}
Introduce a new variable $z$ and consider the ideal $J:=\p\cdot z+\ideal{\det(M)}\cdot (1-z)$ in $R[z]$.
Extending the grevlex ordering to $R[z]$ with $z>x_1>\cdots>z_3$, we may obtain a Gr\"{o}bner basis $\D$  for $J$, consisting of $26$ elements.
By \cite[Algorithm1.2.1]{DK15},  removing all elements of $\D$ that involve the variable $z$ produces  a Gr\"{o}bner basis $\D'$  for $\p\cap \ideal{\det(M)}$, where $\D'=\{\det(M)\cdot f\mid f\in \B\}$. Thus $\p: \det(M)$ can be generated by
$\B$ and it follows from Lemma \ref{lem3.3} that $\p: \det(M)=\p$. Therefore, the class of $\det(M)$ in $R/\p$ is not a zero divisor.
\end{proof}

\begin{prop} \label{prop3.7}
$\p$ is a prime ideal in $R$ of height $3$.
\end{prop}

\begin{proof}
To prove that $\p$ is prime, it suffices to show that $R/\p$ is an integral domain. Solving the linear system of
$f_1=f_2=f_3=0$ obtains that $z_1,z_2,z_3$ can be expressed rationally by other variables. Moreover,
$$(R/\p)\left[\frac{1}{\det(M)}\right]\cong \K[x_1,x_2,x_3,y_1,y_2,y_3]\left[\frac{1}{\det(M)}\right]$$
is an integral domain. By Lemma \ref{lem3.6},  the class of $\det(M)$ is not a zero divisor in $R/\p$, thus $R/\p$ can be embedded into $(R/\p)\left[\frac{1}{\det(M)}\right]$. Hence, $R/\p$ is also an integral domain and $\p$ is prime. 

To compute the height of $\p$, we consider the following algebra homomorphism 
\begin{equation}
\label{ }
\uppi: R\ra \K[x_1,x_2,x_3,y_1,y_2,y_3]\left[\frac{1}{\det(M)}\right]
\end{equation}
defined by $x_i\mapsto x_i, y_i\mapsto y_i,z_1\mapsto 0$, and   
$$z_2\mapsto \frac{1}{\det(M)}(x_1y_3-x_3y_1 - 1), z_3\mapsto \frac{1}{\det(M)}(x_1y_2-x_2y_1).$$ Note that the image of $\uppi$
is an integral domain of Krull dimension $6$, thus it follows from \cite[Theorem 1.8A]{Har77} that
$\ker(\uppi)$ is a prime ideal of height $3$. Clearly, $\p\subseteq \ker(\uppi)$. It follows from Lemma \ref{lem3.5} that
\begin{equation}
\label{ }
0\subset \ideal{f_1}\subset \ideal{f_1,f_2}\subset \p\subseteq \ker(\uppi)
\end{equation}
is a chain of prime ideals in $R$. Hence, $\p=\ker(\uppi)$ and is of height $3$.
\end{proof}

We are ready to prove Theorem \ref{thm2}.

\begin{proof}[Proof of Theorem \ref{thm2}]
By Propositions \ref{prop3.2} and \ref{prop3.7}, we see that $\A_3(\K)$ is a 6-dimensional irreducible affine variety, showing the first statement in Theorem \ref{thm2}. 

To prove the second statement, we need to show that $f_1,f_2,f_3$ is a regular sequence in $R$.
Since $R$ is an integral domain and $f_1\neq 0$, it follows that $f_1$ is not a zero divisor in $R$ and so $f_1$
is regular in $R$. To show that $f_1,f_2$ is regular in $R$, it suffices to show that $f_2$ is not a zero divisor in 
$R/\ideal{f_1}$. By Lemma \ref{lem3.5}, $R/\ideal{f_1}$ is an integral domain. Note that the normal form of $f_2$ with respect to $\{f_1\}$ is nonzero, thus the class of $f_2$ in $R/\ideal{f_1}$ is not a zero divisor. To see that $f_1,f_2,f_3$ is regular in $R$, we need  to show that the class of $f_3$ is not a zero divisor in $R/\ideal{f_1,f_2}$. Note that Lemma \ref{lem3.4} tells us that $\{f_1,f_2\}$ is a Gr\"{o}bner basis of $\ideal{f_1,f_2}$, and using this fact, we may deduce that the normal form of $f_3$ with respect to $\{f_1,f_2\}$ is itself. Thus it is nonzero.  Applying Lemma \ref{lem3.5} again, we see that $R/\ideal{f_1,f_2}$ is an integral domain. Hence, the class of $f_3$ is not a zero divisor in $R/\ideal{f_1,f_2}$. This completes the proof of Theorem 
\ref{thm2}.
\end{proof}

\begin{rem}\label{rem3.8}
{\rm 
Note that the notions of geometrically complete intersections and algebraically complete intersections are, in general, not equivalent. Specifically, an affine variety that is a complete intersection in the geometric sense may have a vanishing ideal that is not a complete intersection ideal, and conversely. However, in our case, $\A_3(\K)$  constitutes a complete intersection in both the geometric and algebraic senses.
\hbo}\end{rem}

\section{Classification of Three-dimensional $\upomega$-Lie Algebras}\label{sec4}
\setcounter{equation}{0}
\renewcommand{\theequation}
{4.\arabic{equation}}
\setcounter{theorem}{0}
\renewcommand{\thetheorem}
{4.\arabic{theorem}}

\noindent This section uses Theorems \ref{thm1} and \ref{thm2} to give a complete classification of all 3-dimensional non-Lie $\upomega$-Lie algebras over an algebraically closed field $\K$ of characteristic $\neq 2$. 

Note that
$\A_3(\K)=\LL_\upomega(\K)$, where we may assume that 
\begin{equation}
\label{ }
\upomega(x,y)=1\textrm{ and } \upomega(x,z)=\upomega(y,z)=0.
\end{equation}
Recall that the symplectic group $\Sp_2(\K)=\SL_2(\K)$; see \cite[Theorem 8.1]{Tay92}. Thus, the corresponding 
stabilizer of $\upomega$ in $\GL_3(\K)$ is 
\begin{equation}
\label{eq4.2}
G_\upomega=\left\{\begin{pmatrix}
  S    &  0  \\
   T   &  d
\end{pmatrix}\mid S\in\SL_2(\K), T\in \K^2,  d\in\K^\times\right\}.
\end{equation}
Throughout this subsection, unless otherwise stated,  we always suppose that  $g$ denotes a generic element in $G_\upomega$ such that 
\begin{equation}
\label{generic}
g^{-1}=\begin{pmatrix}
    s_1  &s_3&0    \\
     s_2 &s_4&0\\
     t_1&t_2&d  
\end{pmatrix}
\end{equation}
where $\begin{pmatrix}
      s_1& s_3   \\
      s_2& s_4 
\end{pmatrix}\in\SL_2(\K)$. Hence,
\begin{equation}
\label{dsl2}
s_1s_4-s_2s_3=1.
\end{equation}

Let $L\in\LL_\upomega(\K)$ be a generic 3-dimensional $\upomega$-Lie algebra, spanned by $\{x,y,z\}$. Then the action of $g^{-1}$ on $L$ can be given by
\begin{equation}\label{action2}
\begin{aligned}
g^{-1}(x) & =  s_1x+s_2y+t_1z \\
g^{-1}(y) & =  s_3x+s_4y+t_2z\\
g^{-1}(z)&=dz.
\end{aligned}
\end{equation}
Moreover, we assume that the brackets in $L$ as follows:
\begin{equation}\label{eq4.6}
\begin{aligned}
[x,y] & =~  a_1x+a_2y+a_3z \\
[x,z] & =~  b_1x+b_2y+b_3z \\
[y,z] & =~  c_1x+c_2y+c_3z.
\end{aligned}
\end{equation}
By Proposition \ref{prop3.2}, we see that
\begin{equation}\label{eq4.7}
\begin{aligned}
a_2b_1 - a_1b_2 - b_3c_2 + b_2c_3&=0\\
a_2c_1 + b_3c_1 - a_1c_2 - b_1c_3&=0 \\
a_3b_1 - a_1b_3 + a_3c_2 - a_2c_3 + 1&=0.
\end{aligned}
\end{equation}

Combining the third equation in (\ref{eq4.7}) with the first equation in (\ref{eq4.6}), we see that $[x,y]\neq 0$. Thus, the arguments in this section are divided into two cases: $a_3\neq 0$ and $a_3=0$. We define
\begin{equation}
\label{delta}
\Delta:= b_1c_2-b_2c_1
\end{equation}
which plays a key role in our arguments below.

\subsection{The case $a_3\neq 0$ and $\Delta \neq 0$} 

In this case, the following Lemma \ref{lem4.1} demonstrate that without loss of generality, the first equation in (\ref{eq4.6}) actually
can be assumed as 
\begin{equation}
\label{eq4.9}
[x,y]=z
\end{equation}
because any two elements in the same orbit of $\LL_\upomega(\K)$ under the action of $G_\upomega$ produce  isomorphic 
$\upomega$-Lie algebras, by Theorem \ref{thm1}. 

To show Lemma \ref{lem4.1}, we note that  replacing $z$ by $a_3^{-1}z$ and fixing $x$ and $y$  in (\ref{eq4.6}) obtain an isomorphic $\upomega$-Lie algebra. Thus we may assume that $a_3=1$ in (\ref{eq4.6}). In fact, using the group element 
$$g=\begin{pmatrix}
    1  &0&0    \\
     0&1&0\\
     0&0&a_3^{-1} 
\end{pmatrix}\textrm{ and }g^{-1}=\begin{pmatrix}
    1  &0&0    \\
     0&1&0\\
     0&0&a_3
\end{pmatrix}\in G_\upomega,$$
we see that $[x,y]_g=g[g^{-1}(x),g^{-1}(y)]=g[x,y]=g(a_1x+a_2y+a_3z)=
a_1g(x)+a_2g(y)+a_3g(z)=a_1x+a_2y+a_3(a_3^{-1} z)=a_1x+a_2y+z$. As $b_i$ and $c_i$
are free variables currently, so $[x,z]_g$ and $[y,z]_g$ can be assumed to have the same forms as in 
(\ref{eq4.6}).

\begin{lem}\label{lem4.1}
There exists an element $g\in G_\upomega$ such that $[x,y]_g=z$, where $[x,y]_g$ is defined as {\rm (}\ref{action}{\rm)}.
\end{lem}

\begin{proof}
Suppose that  $g$ denotes a generic element in $G_\upomega$ such that  $g^{-1}$
has the form in (\ref{generic}). Clearly, the statement that $z=[x,y]_g=g([g^{-1}(x), g^{-1}(y)])$ holds if and only if 
\begin{equation}
\label{eq4.10}
g^{-1}(z)=[g^{-1}(x), g^{-1}(y)].
\end{equation}
Thus, it suffices to verify (\ref{eq4.10}) for some $g\in G_{\upomega}$.  By (\ref{action2}), we note that $g^{-1}(z)=dz$, as well as
\begin{eqnarray*}
&&[g^{-1}(x), g^{-1}(y)]=[s_1x+s_2y+t_1z, s_3x+s_4y+t_2z]\\
&=&[x,y]+(s_1t_2-s_3t_1)[x,z]+(s_2t_2-s_4t_1)[y,z]\\
&=&(a_1x+a_2y+z)+(s_1t_2-s_3t_1)(b_1x+b_2y+b_3z)+(s_2t_2-s_4t_1)(c_1x+c_2y+c_3z).
\end{eqnarray*}
By the linear independence of $x,y,z$, we observe that (\ref{eq4.10}) holds if and only if 
\begin{eqnarray*}
a_1+(s_1t_2-s_3t_1)b_1+(s_2t_2-s_4t_1)c_1 & = & 0 \\
a_2+(s_1t_2-s_3t_1)b_2+(s_2t_2-s_4t_1)c_2 & = & 0 \\
1+(s_1t_2-s_3t_1)b_3+(s_2t_2-s_4t_1)c_3 & = & d.
\end{eqnarray*}
To determine $g^{-1}$, we treat $t_1,t_2$ and $d$ as indeterminates, the above linear system can be rewritten as 
\begin{equation}\label{eq4.11}
\begin{aligned}
(s_3b_1+s_4c_1)t_1-(s_1b_1+s_2c_1)t_2& =  a_1 \\
(s_3b_2+s_4c_2)t_1-(s_1b_2+s_2c_2)t_2 & =  a_2 \\
(s_3b_3+s_4c_3)t_1-(s_1b_3+s_2c_3)t_2+d & =  1.
\end{aligned}
\end{equation}
We write $T$ for the coefficient matrix of (\ref{eq4.11}) on the left-hand side. Note that 
$\Delta\neq 0$. Thus
\begin{equation}\label{ }
\begin{aligned}
\det(T)&=\Big((s_1b_1+s_2c_1)(s_3b_2+s_4c_2)-(s_3b_1+s_4c_1)(s_1b_2+s_2c_2)\Big)\cdot d\\
&=(s_1s_4-s_2s_3)(b_1c_2-b_2c_1)\cdot d\\
&\hspace{-1.7mm}\oeq{dsl2} (b_1c_2-b_2c_1)\cdot d\oeq{delta} \Delta \cdot d\neq 0.
\end{aligned}
\end{equation}
By linear algebra, $(t_1,t_2,d)^t=T^{-1}\cdot (a_1,a_2,1)^t$ is the unique solution, which
determines an element $g\in G_\upomega$ such that $[x,y]_g=z$. 
\end{proof}

\begin{rem}{\rm
In the proof of Lemma \ref{lem4.1}, we see that the vector $(t_1,t_2,d)$ is uniquely determined by $s_1,s_2,s_3$, and $s_4$.  Thus, the set of all elements in $G_\upomega$ such that $[x,y]_g=z$ is parameterized by $\SL_2(\K)$.
\hbo}\end{rem}

By Lemma \ref{lem4.1}, we may assume that $a_1=a_2=0$ and $a_3=1$ in (\ref{eq4.6}), which leads the equations in (\ref{eq4.7}) become
\begin{equation}\label{eq4.13}
\begin{aligned}
b_3c_2 - b_2c_3&=0\\
b_3c_1 - b_1c_3&=0 \\
b_1  + c_2  + 1&=0.
\end{aligned}
\end{equation}
The first two equations in (\ref{eq4.13}) can be written as
$\begin{pmatrix}
    c_2  & -b_2   \\
     c_1 &  -b_1
\end{pmatrix}\begin{pmatrix}
      b_3    \\
      c_3 
\end{pmatrix}=\begin{pmatrix}
      0    \\
      0  
\end{pmatrix}$
which implies that $b_3=c_3=0$, because of $\Delta\neq 0$. 

Therefore, the $\upomega$-Lie algebra $L$ has the following brackets:
\begin{equation}\label{eq4.14}
\begin{aligned}
[x,y] & =~  z \\
[x,z] & =~  b_1x+b_2y \\
[y,z] & =~  c_1x+c_2y.
\end{aligned}
\end{equation}

To eliminate more redundant variables in $b_i$ and $c_i$ while keep (\ref{eq4.9}) preserved,  we write $H_\upomega$ for the stabilizer of $[x,y]=z$ in $G_\upomega$, i.e.,
\begin{equation}
\label{eq4.15}
H_\upomega:=\{g\in G_\upomega\mid [x,y]_{g}=z\}=\{g\in G_\upomega\mid [g^{-1}(x),g^{-1}(y)]=g^{-1}(z)\}.
\end{equation}
The following lemma determines the structure of $H_\upomega$. 

\begin{lem}\label{lem4.3}
$H_\upomega=\left\{\begin{pmatrix}
  S    &  0  \\
   0   &  1
\end{pmatrix}\mid S\in\SL_2(\K)\right\}.$
\end{lem}

\begin{proof}
As in the proof of Lemma \ref{lem4.1}, $g$ denotes a generic element in $G_\upomega$ such that  $g^{-1}$
has the form in (\ref{generic}). The condition $[g^{-1}(x),g^{-1}(y)]=g^{-1}(z)$, which is equivalent to $[x,y]_g=z$, can be read as (\ref{eq4.11}). Note that $a_1=a_2=b_3=c_3=0$. Thus (\ref{eq4.11}) becomes
\begin{equation}\label{eq4.16}
\begin{aligned}
(s_3b_1+s_4c_1)t_1-(s_1b_1+s_2c_1)t_2& =  0 \\
(s_3b_2+s_4c_2)t_1-(s_1b_2+s_2c_2)t_2 & =  0 \\
d & =  1.
\end{aligned}
\end{equation}
Clearly, $(t_1,t_2,d)=(0,0,1)$ is the unique solution of this linear system. Hence, $g\in H_\upomega$ if and only if
$g=\begin{pmatrix}
  S    &  0  \\
   0   &  1
\end{pmatrix}$ for some $S\in\SL_2(\K)$.
\end{proof}

\begin{rem}{\rm
By Lemma \ref{lem4.3}, the action of $H_\upomega$ on $(L,[-,-])$ preserves the relation $[x,y]=z$ in (\ref{eq4.14})
but it might change the two relations $[x,z]$ and $[y,z]$ in (\ref{eq4.14}). We use
$\ad$ and $\widetilde{\ad}$ to denote the adjoint maps in $(L,[-,-])$ and $(L,[-,-]_g)$ respectively. 
For each $g\in H_\upomega$, we see that $[x,z]_g=-\widetilde{\ad}_z(x)$
and $g[g^{-1}(x),g^{-1}(z))]=-g(\ad_{g^{-1}(z)}(g^{-1}(x)))=-(g\circ\ad_z\circ g^{-1})(x)$. Thus 
$$\widetilde{\ad}_z(x)=(g\circ\ad_z\circ g^{-1})(x).$$
Similarly, $\widetilde{\ad}_z(y)=(g\circ\ad_z\circ g^{-1})(y).$ Thus, on the subspace $\ideal{x,y}$, $\widetilde{\ad}_z$ and
$\ad_z$ are conjugate in $H_{\upomega}$. By (\ref{eq4.14}), it follows that up to sign, 
$$\ad_z=\begin{pmatrix}
    b_1  &   c_1 \\
    b_2  &  c_2
\end{pmatrix}.$$
Note that the third equation in (\ref{eq4.13}) implies that $b_1+c_2=-1$. Therefore, classifying
all $\upomega$-Lie algebras in $\LL_\upomega(\K)$ in this case ($a_3\neq 0$ and $\Delta\neq 0$) is equivalent to
classifying the canonical forms of all $2\times 2$ matrices of  trace $-1$ over $\K$ under the conjugacy action of $\SL_2(\K)$.
\hbo}\end{rem}

We write $\M(\K)$ for the variety of all $2\times 2$ matrices of  trace $-1$ over $\K$ on which 
$\SL_2(\K)$ acts by conjugation. Recall that the conjugation action preserves trace and determinant. 

Suppose that $M\in \M(\K)$ denotes a generic matrix. Then the characteristic polynomial 
$$\upchi_{M}(t)=t^2-\trace(M)\cdot t+\det(M)=t^2+ t+\det(M)$$
which only depends on $\det(M)$. Since $\K$ is algebraically closed, $\upchi_{M}(t)=(t-b)(t-c)$ splits in $\K$ for some
$b,c\in\K$. Note that $b+c=\trace(M)=-1$, thus $c=-(b+1)$. Clearly, if $b=-\frac{1}{2}$, then
$c=-\frac{1}{2}$. Thus $\upchi_{M}(t)=0$ has one root in $\K$. When $b\neq -\frac{1}{2}$, $b\neq c$ and
$\upchi_{M}(t)=0$ has two distinct roots  in $\K$.

\begin{lem}\label{lem4.5}
If $b\neq -\frac{1}{2}$, then $M$ and $\dia\{b,-(b+1)\}$ have the same conjugacy orbit. In particular, there exists only
one conjugacy class in $\M(\K)$.
\end{lem}

\begin{proof}
Note that $b\neq c=-(b+1)$ are two distinct roots of $\upchi_{M}(t)=0$ in $\K$. Thus the Jordan canonical form
of $M$ is $\dia\{b,c\}$. In other words, $M$ and $\dia\{b,c\}$ are conjugate in $\GL_2(\K)$. By Lemma \ref{lem4.6} below,
$M$ and $\dia\{b,c\} (= \dia\{b,-(b+1)\})$ are also conjugate in $\SL_2(\K)$. 

Moreover, note that $\dia\{b,c\}$ and $\dia\{c,b\}$ are conjugate in $\SL_2(\K)$ via the matrix $J$ in (\ref{eq2.4}). Thus the
 two scalars $b$ and $c$ together define one conjugacy class. As $c=-(b+1)$, we see that 
there exists a unique conjugacy class in $\M(\K)$, represented by $\dia\{b,-(b+1)\}$. 
\end{proof}

\begin{lem}\label{lem4.6}
Let $A,B$ be two matrices of size $2$ over $\K$. Then they are conjugate in $\SL_2(\K)$ if and only if
they are conjugate in $\GL_2(\K)$.
\end{lem}

\begin{proof}
$(\RA)$ It is immediate. $(\LA)$ Since $A$ and $B$ are conjugate in $\GL_2(\K)$, we may suppose that $B=g^{-1}Ag$ for some $g\in\GL_n(\K)$. Note that $\K$ is algebraically closed, thus there exists a scalar $c\in\K^\times$ such that $c^2=\det(g)$.
Define $h:=c^{-1}\cdot g$. Then $\det(h)=\det(c^{-1}\cdot g)=c^{-2}\cdot \det(g)=\det(g)^{-1}\cdot \det(g)=1$, which means that
$h\in\SL_2(\K)$. Moreover, the fact that $g=c\cdot h$ implies that 
$$B=g^{-1}Ag=(c\cdot h)^{-1}A(c\cdot h)=h^{-1}Ah.$$
Hence, $A$ and $B$ are conjugate in $\SL_2(\K)$ as well.
\end{proof}

\begin{rem}{\rm
If $\K$ is not algebraically closed, then this result can be extended to the statement that $A$ and $B$ are conjugate in $\SL_2(\K)$ if and only if they are conjugate in $\GL_2(\K)$ via some $g$ such that $\det(g)$ is a square in $\K^\times$.
\hbo}\end{rem}

By Lemma \ref{lem4.5}, we may choose 
$$M=\begin{pmatrix}
    b_1  &   c_1 \\
    b_2  &  c_2
\end{pmatrix}=\begin{pmatrix}
     b &   0 \\
     0 &  -(b+1)
\end{pmatrix}.$$
Note that $\det(M)=\Delta\neq 0$. Thus we derive a family of $\upomega$-Lie algebras:
\begin{equation}\tag{$C^*_\upalpha$}\label{C_a}
[x,y]=z, [x,z]=\upalpha x,[y,z]=-(\upalpha+1)y,\textrm{ where }\upalpha\in\K\setminus\{-1,-\frac{1}{2},0\}.
\end{equation}

\begin{rem}{\rm
(1) We will see from (\ref{eq4.17}) and (\ref{eq4.22}) below that when $\upalpha\in\{-1,-\frac{1}{2}\}$, 
the relations in (\ref{C_a}) remains to give an $\upomega$-Lie algebra (or a direct verification). In other words, 
the condition in  (\ref{C_a}) can be partially removed. Hence, we obtain a family of $\upomega$-Lie algebras:
\begin{equation}\tag{$C_\upalpha$}\label{C}
[x,y]=z, [x,z]=\upalpha x,[y,z]=-(\upalpha+1)y,\textrm{ where }\upalpha\in\K^\times.
\end{equation}
Moreover,
when $\upalpha=0$, $C_0$ is actually an  $\upomega$-Lie algebra but isomorphic to $C_{-1}$ by the discussion in the first subcase of Section 4.2 later.

(2) When $\K=\C$, this family is isomorphic to the family $C_\upalpha$ appeared in \cite[Theorem 2]{CLZ14}. Further, 
the $\upomega$-Lie algebra $L_2$ appeared in \cite[Theorem 2]{CLZ14} is also isomorphic to $C_0\cong C_{-1}$.
\hbo}\end{rem}

Now let us consider the case where $b=-\frac{1}{2}$. 

 \begin{lem}\label{lem4.9}
If $b= -\frac{1}{2}$, then  there are two conjugacy classes in $\M(\K)$, represented by 
$$\begin{pmatrix}
    b  &  0  \\
    0  &  b
\end{pmatrix}\textrm{ and }\begin{pmatrix}
    b  &  1  \\
    0  &  b
\end{pmatrix}$$ respectively.
\end{lem}

\begin{proof}
If $b= -\frac{1}{2}$, then $b$ is the unique root of $\upchi_{M}(t)=0$ in $\K$. The Jordan canonical form of $M$ only has two possibilities: $\begin{pmatrix}
    b  &  0  \\
    0  &  b
\end{pmatrix}\textrm{ and }\begin{pmatrix}
    b  &  1  \\
    0  &  b
\end{pmatrix}.$ By Lemma \ref{lem4.6}, $M$ is conjugate to one of them in $\SL_2(\K)$.
\end{proof}

By Lemma \ref{lem4.9}, when $M=\begin{pmatrix}
    b_1  &   c_1 \\
    b_2  &  c_2
\end{pmatrix}=\begin{pmatrix}
     -\frac{1}{2} &   0 \\
     0 & -\frac{1}{2}
\end{pmatrix}$, we obtain the following $\upomega$-Lie algebra:
\begin{equation}
\label{eq4.17}
[x,y]=z, [x,z]=-x/2, [y,z]=-y/2
\end{equation}
which fixes the case where $\upalpha\neq-\frac{1}{2}$ in (\ref{C_a}).

Moreover, setting $M=\begin{pmatrix}
    b_1  &   c_1 \\
    b_2  &  c_2
\end{pmatrix}=\begin{pmatrix}
     -\frac{1}{2} &   1 \\
     0 & -\frac{1}{2}
\end{pmatrix}$
gives rise to
\begin{equation}\tag{$B$}\label{B}
[x,y]=z, [x,z]=-x/2, [y,z]=x-y/2
\end{equation}
which is isomorphic to the  $\upomega$-Lie algebra $B$ occurred  in \cite[Theorem 2]{CLZ14} when $\K=\C$.

\subsection{The case $a_3\neq0$ and $\Delta=0$} 

Since $\Delta=b_1c_2-b_2c_1=0$, it follows that $(b_1,b_2)$ and $(c_1,c_2)$
are linearly dependent over $\K$. We may assume that the brackets in $L$ as follows:
\begin{equation}\label{eq4.18}
\begin{aligned}
[x,y] & =~  a_1x+a_2y+a_3z \\
[x,z] & =~  b_1x+b_2y+b_3z \\
[y,z] & =~  cb_1x+cb_2y+c_3z
\end{aligned}
\end{equation}
for some $c\in\K$. 

Replacing $y$ with $y-cx$ and fixing $x$ and $z$,  the third equation in (\ref{eq4.18}) can be assumed as 
\begin{equation}\label{ }
[y,z]=c_3z.
\end{equation}
In fact, we may choose
$$g=\begin{pmatrix}
    1  &c&0    \\
     0&1&0\\
     0&0&1
\end{pmatrix}\textrm{ and }g^{-1}=\begin{pmatrix}
    1  &-c&0    \\
     0&1&0\\
     0&0&1
\end{pmatrix}\in G_\upomega,$$
and observe that 
\begin{eqnarray*}
[y,z]_g&=&g[g^{-1}(y),g^{-1}(z)]=g[y-cx,z]\\
&=&g(cb_1x+cb_2y+c_3z-c(b_1x+b_2y+b_3z))\\
&=& (c_3-cb_3)z.
\end{eqnarray*}
As $a_i$ and $b_i$ are free variables, $[x,y]_g$ and $[x,z]_g$
can be assumed to have the forms in  (\ref{eq4.18}).

This leads to two subcases: $c_3=0$ and $c_3\neq 0$. 

\textsc{Subcase 1:} $c_3=0$. In this case,  (\ref{eq4.7}) becomes
\begin{equation}\label{eq4.20}
a_2b_1 - a_1b_2 =0\textrm{ and } a_3b_1 - a_1b_3+ 1=0
\end{equation}
which shows that $(a_1,a_2)$ and $(b_1,b_2)$are linearly dependent over $\K$.
We may assume that $$[x,y] =~  cb_1x+cb_2y+a_3z$$ for some $c\in\K$.
Replacing $y$ with $y-cz$ and fixing $x$ and $z$, i.e., applying the group element
$$g=\begin{pmatrix}
    1  &0&0    \\
     0&1&0\\
     0&c&1
\end{pmatrix}\in G_\upomega,$$
together with the second equation in (\ref{eq4.18}) and $[y,z]=0$, we may assume that the brackets in $L$ are given as follows:
\begin{equation}
\label{eq4.21}
[x,y]=a_3z, [x,z] =  b_1x+b_2y+b_3z,\textrm{ and }[y,z]=0.
\end{equation}
Note that $a_3\neq 0$. Thus we may take $a_3=1$ via replacing $z$ by $a_3^{-1}z$ and furthermore, it follows from (\ref{eq4.20}) that $b_1=-1$.
Thus, (\ref{eq4.21}) turns out that
\begin{equation*}
[x,y]=z, [x,z] =-x+b_2y+b_3z,\textrm{ and }[y,z]=0.
\end{equation*}
Replacing $x$ with $x-b_2y-b_3z$ and fixing $y$ and $z$, i.e., applying the group element $g\in G_\upomega$ such that
$$g^{-1}=\begin{pmatrix}
    1  &0&0    \\
    - b_2&1&0\\
    - b_3&0&1
\end{pmatrix},$$
we derive the following $\upomega$-Lie algebra:
\begin{equation}\label{eq4.22}
[x,y]=z, [x,z] =-x, [y,z]=0.
\end{equation}
which fixes the case where $\upalpha\neq-1$ in (\ref{C_a}).

\textsc{Subcase 2:} $c_3\neq 0$. Scaling $z$ by $c_3^{-1}$ allows us to assume that
$c_3=1$ in this case. By  (\ref{eq4.7}), it follows that $b_1=0$ and
\begin{equation}
\label{eq4.23}
(1-a_1)b_2=0\textrm{ and }a_1b_3+a_2=1.
\end{equation}
Moreover, (\ref{eq4.18}) becomes
\begin{equation}\label{eq4.24}
\begin{aligned}
[x,y] & =~  a_1x+a_2y+a_3z \\
[x,z] & =~  b_2y+b_3z \\
[y,z] & =~  z.
\end{aligned}
\end{equation}
Replacing $x$ with $x-b_3y$ and fixing $y$ and $z$ allows us to assume that $b_3=0$ in (\ref{eq4.24}). 
Together with (\ref{eq4.23}) implies that $a_2=1$.  As $(1-a_1)b_2=0$, either $b_2=0$
or $a_1=1$. For the first subcase ($b_2=0$), we see that:
\begin{equation*}
\label{ }
[x,y]=a_1x+y+a_3z, [x,z]=0, [y,z]=z.
\end{equation*}
Replacing $y$ with $a_1x+y+a_3z$ and fixing $x$ and $z$ obtains
\begin{equation}\tag{$D$} 
\label{D}
[x,y]=y, [x,z]=0, [y,z]=z
\end{equation}
which is  the  $\upomega$-Lie algebra $L_1$  in \cite[Theorem 2]{CLZ14} if $\K=\C$.

For the latter subcase ($a_1=1$), we have
\begin{equation*}
\label{ }
[x,y]=x+y+a_3z, [x,z]=b_2y, [y,z]=z
\end{equation*}
where $b_2\neq 0$, because if $b_2=0$, we obtain an $\upomega$-Lie algebra, which is isomorphic to (\ref{D}).
Replacing $z$ by $b_2^{-1}z$ and fixing $x$ and $y$, leads us to assume that $L$ has the following brackets:
\begin{equation}
\label{eq4.25}
[x,y]=x+y+a_3z, [x,z]=y, [y,z]=z.
\end{equation}
The above process clearly can be realized via applying some elements in $G_\upomega$.

\begin{lem}
Let $N_\upomega$ be the stabilizer of $[y,z]=z$ and $[x,z]=y$  in $G_\upomega$. Then $N_\upomega$ consists of all $g\in G_\upomega$ such that
\begin{equation}
\label{eq4.26}
g^{-1}=\begin{pmatrix}
    1  &  0&0  \\
     s & 1&0\\
     t&s&1 
\end{pmatrix}.
\end{equation}
\end{lem}

\begin{proof}
Suppose that $g\in N_\upomega$ denotes a generic element such that $g^{-1}$ has
a form as in (\ref{generic}). Then $[y,z]_g=z$, which is equivalent to $[g^{-1}(y),g^{-1}(z)]=g^{-1}(z).$ By (\ref{action2}), this means that $$[s_3x+s_4y+t_2z, dz]=dz.$$
Note that $d\neq 0$. Thus,  $s_3y+s_4z=z$. This means that $s_3=0$ and $s_4=1$. 
Moreover, since $\begin{pmatrix}
    s_1  & s_3   \\
     s_2 & s_4 
\end{pmatrix}$ has determinant $1$, it follows that $s_1=1$.  Hence, $g^{-1}$ has the following form:
$$g^{-1}=\begin{pmatrix}
    1  &  0&0  \\
     s_2 & 1&0\\
     t_1&t_2&d
\end{pmatrix}.$$
Since $g$ also preserves $[x,z]=y$, we see that $[g^{-1}(x),g^{-1}(z)]=g^{-1}(y)$, which can be read as
$$[x+s_2y+t_1z,dz]=y+t_2z.$$
Namely, $dy+s_2dz=y+t_2z.$ Thus $d=1$ and $t_2=s_2$. This completes the proof.
\end{proof}

To complete the argument of this case, we consider the action of $g^{-1}$ on $x$ and $y$, where $g\in N_\upomega$ such that $g^{-1}$ has a form as in (\ref{eq4.26}). Clearly, 
$$g=\begin{pmatrix}
    1  &  0&0  \\
     -s & 1&0\\
     s^2-t&-s&1
\end{pmatrix}.$$
Note that
\begin{eqnarray*}
[g^{-1}(x),g^{-1}(y)] & = & [x+sy+tz, y+sz] = (x+y+a_3z)+sy+s^2z-tz \\
 &=&x+(1+s)y+(a_3+s^2-t)z.
\end{eqnarray*}
Thus
\begin{eqnarray*}
[x,y]_g&=&g[g^{-1}(x),g^{-1}(y)] =g(x)+(1+s)g(y)+(a_3+s^2-t)g(z)\\
&=&(x-sy+(s^2-t)z)+(1+s)(y-sz)+(a_3+s^2-t)z\\
&=&x+y+(s^2-t-s^2-s+a_3+s^2-t)z\\
&=&x+y+(s^2-2t-s+a_3)z.
\end{eqnarray*}
Define $t=\frac{1}{2}(s^2-s+a_3)$. We obtain a $g\in N_{\upomega}$ such that 
$$[x,y]_g=x+y.$$
Therefore, this gives rise to
\begin{equation}\tag{$A$} 
\label{A}
[x,y]=x+y, [x,z]=y, [y,z]=z.
\end{equation}
which is isomorphic to the  $\upomega$-Lie algebra $A_0$  in \cite[Theorem 2]{CLZ14} if $\K=\C$.
This also shows that $A_{\upalpha}\cong A_{0}$  in \cite[Theorem 2]{CLZ14} for all $\upalpha\in\C$, which is isomorphic to $A$ above.

\subsection{The case $a_3=0$} 

Note that $[x,y]\neq 0$. Thus in this case, we see that either $a_1\neq 0$ or $a_2\neq 0$. 
Without loss of generality, we may assume that $a_1\neq 0$ because for the case $a_2\neq 0$, we will obtain 
the same $\upomega$-Lie algebras. Replacing $x$ and $y$ with $a_1^{-1}x$ and $a_1y$ respectively, we actually can assume that $a_1=1$. Thus the brackets in $L$ are: 
\begin{equation}\label{eq4.27}
\begin{aligned}
[x,y] & =~  x+a_2y \\
[x,z] & =~  b_1x+b_2y+b_3z \\
[y,z] & =~  c_1x+c_2y+c_3z.
\end{aligned}
\end{equation}
Replacing $x$ by $x+a_2y$, we may assume that $a_2=0$ in (\ref{eq4.27}). 
Then (\ref{eq4.7}) becomes
\begin{equation}\label{eq4.28}
\begin{aligned}
 b_2 + b_3c_2 - b_2c_3&=0\\
  b_3c_1 - c_2 - b_1c_3&=0 \\
 b_3   - 1&=0.
\end{aligned}
\end{equation}
Since $b_3=1$, we may replace $y$ by $y-c_3x$ and assume that $c_3=0$ in (\ref{eq4.27}). 
Substituting this assumption back to  (\ref{eq4.28}), we see that
$$ b_2 + c_2=0\textrm{ and } c_1 - c_2 =0.$$
Clearly, these base changes above can equivalently to be realized via applying elements in $G_\upomega$. 
Hence, we may assume that $L$ has the following brackets:
\begin{equation} \label{eq4.29}
\begin{aligned}
[x,y] & =~  x \\
[x,z] & =~  b_1x-c_1y+z \\
[y,z] & =~  c_1x+c_1y.
\end{aligned}
\end{equation}

\textsc{Subcase 1}: $c_1=0$. If $b_1=0$, we derive from (\ref{eq4.29}) the following $\upomega$-Lie algebra:
$$[x,y]=x, [x,z]=z, [y,z]=0$$
which is isomorphic to (\ref{D}).  If $b_1\neq 0$, then $[x,z]=b_1x+z$.  Replacing 
$x$ by $x-z$, we obtain 
\begin{equation}
[x,y]=x+z, [x,z]=bx+(b+1)z, [y,z]=0
\end{equation}
which is isomorphic to the  $\upomega$-Lie algebra $C_{-1}$ by discussion in Subsection 4.2.

\textsc{Subcase 2}: $c_1\neq0$. In this case, replacing $z$ by $c_1^{-1}z$, we obtain
\begin{equation}
\label{ }
[x,y]=x, [x,z]=bx-y+z, [y,z]=x+y.
\end{equation}
Replacing $x$ by $x-z$, we obtain 
$$[x,y]=2x+y+z,[x,z]=bx-y+(b+1)z,[y,z]=x+y+z$$
which will be isomorphic to one of $\{A,B,C_\upalpha,D\mid \upalpha\in\K^\times\}$ by Subsections 4.1
and 4.2.

\subsection{Proof of Theorem \ref{thm3}}

Together with the previous Subsections 4.1, 4.2, and 4.3, we can give a proof to Theorem \ref{thm3}.

\begin{proof}[Proof of the first statement of Theorem \ref{thm3}]
We assume that $\upomega(x,y)=1$ and $\upomega(x,z)=\upomega(y,z)=0$ for $L\in \A_3(\K)$. If
the $z$-component of $[x,y]$ is nonzero, by Subsections 4.1 and 4.2, we see that $L$ must be isomorphic to one of $\{A,B,C_\upalpha,D\mid \upalpha\in\K^\times\}$. If the $z$-component of $[x,y]$ is zero, then by the discussion in this subsection, there exists a group element $g\in G_\upomega$ such that the $z$-component of the new bracket $[x,y]_g$ is nonzero. Hence, $L$ is isomorphic to one of $\{A,B,C_\upalpha,D\mid \upalpha\in\K^\times\}$ as well.
\end{proof}

\begin{proof}[Proof of the second statement of Theorem \ref{thm3}]
Let us begin with computing the dimension of the derived algebra $[L,L]$. Clearly, $\dim([D,D])=\dim([C_{-1},C_{-1}])=2$ and others are $3$. Thus $D$ is non-isomorphic to any one of $\{A,B,C_{\upalpha}\mid \upalpha\in\K\setminus\{0,-1\}\}$, and there is only one possibility that $D\cong C_{-1}$.  Assume by way of contradiction that they are isomorphic. By Theorem \ref{thm1}, they share the same orbit, i.e.,  there exists an invertible element $g\in G_\upomega$ such that $C_{-1}=g\cdot D$ and $g^{-1}$ has the form as in (\ref{generic}). Recall that $D: [x,y]=y, [x,z]=0, [y,z]=z$ and $C_{-1}: [x,y]=z, [x,z]=-z,[y,z]=0$.
Note that $0=[y,z]_g=g[g^{-1}(y),g^{-1}(z)]=g([s_3x+s_4y+t_2z,dz])=g(s_4dz)=(s_4d)(d^{-1})z=s_4z$. Thus
$s_4=0$. Similarly, as $-z=[x,z]_g=g[g^{-1}(x),g^{-1}(z)]=g([s_1x+s_2y+t_1z,dz])$, we see that $s_2=-1$.
Now $z=[x,y]_g=g[g^{-1}(x),g^{-1}(y)]=g[s_1x+s_2y+t_1z, s_1x+s_2y+t_1z]=g(y-t_2z)$. Hence, $g^{-1}(z)=y-t_2z$ and
$y=g^{-1}(z)+t_2z=(d+t_2)z\in\ideal{z}$. This contradiction shows that $D$ and $C_{-1}$ are not isomorphic. 

We need to show that any two of $\{A,B,C_{\upalpha}\mid \upalpha\in\K^\times\}$ are non-isomorphic. 
Assume by way of contradiction that there exists an invertible element $g\in G_\upomega$ such that $B=g\cdot A$. Then $z=[y,z]_g=g[g^{-1}(y),g^{-1}(z)]=g[s_3x+s_4y+t_2z,dz]=g(v)$, where $v\in\ideal{x,y}$. Thus $v=g^{-1}(z)=dz$ and
$z=d^{-1}z\in \ideal{x,y}$. This contradiction shows that $A$ and $B$ are not isomorphic. A same argument shows that $A$ and $C_{\upalpha}$ are not isomorphic for all $\upalpha\in\K^\times$. We only need to focus on $B$ and $C_{\upalpha}$. However, by Subsection 4.1, $B$ and $C_{\upalpha}$ correspond to different conjugate canonical forms under the action of $\SL_2(\K)$. Thus $B$ and $C_{\upalpha}$ are not isomorphic for all $\upalpha\in\K^\times$.

Hence, it suffices to shows that $C_{\upalpha}$ and $C_{\upbeta}$ are not isomorphic for any $\upalpha\neq\upbeta$ in $\K^\times$. Recall that $C_{\upalpha}: [x,y]=z,[x,z]=\upalpha x, [y,z]=-(\upalpha+1)z$. Assume by way of contradiction that there exists an invertible element $h\in H_\upomega$ defined in (\ref{eq4.15}) such that $C_{\upbeta}=h\cdot C_{\upalpha}$. Thus,
$$\upbeta x=[x,z]_h=h[h^{-1}(x),h^{-1}(z)]=h[s_3x+s_4y,z]=h(s_3\upalpha x-s_4(\upalpha+1)z).$$
This means that $\upbeta(s_3x+s_4y)=s_3\upalpha x-s_4(\upalpha+1)z$, which implies that $s_3=s_4=0$. However,
$$\det\begin{pmatrix}
    s_1  &  s_3  \\
     s_2 &  s_4
\end{pmatrix}=1.$$
This contradiction shows that $C_{\upalpha}$ and $C_{\upbeta}$ are not isomorphic when $\upalpha\neq\upbeta$ in $\K^\times$.
\end{proof}

\section{Two Remarks}\label{sec5}
\setcounter{equation}{0}
\renewcommand{\theequation}
{5.\arabic{equation}}
\setcounter{theorem}{0}
\renewcommand{\thetheorem}
{5.\arabic{theorem}}

\noindent This last section contains two remarks about constructions of finite-dimensional $\upomega$-Lie algebras over any field $\K$ (not necessarily algebraically closed) of characteristic $\neq 2$. 

For the first remark, let us consider the 3-dimensional case and write 
$$C=\begin{pmatrix}
     x_1 &x_2&x_3    \\
     y_1 &y_2&y_3    \\ 
     z_1 &z_2&z_3    \\
\end{pmatrix}\in M_3(R)$$
for the matrix corresponding to the structure constants in (\ref{eq3.1}) of a generic $\upomega$-Lie algebra $L$ over 
a field $\K$ such as the real numbers $\R$ or a finite field $\F_q$.
Via the prime ideal $\p=\ideal{f_1,f_2,f_3}$, we may specialize $C$ to directly produce or recover all non-Lie 3-dimensional  $\upomega$-Lie algebra structures, where $f_1,f_2$ and $f_3$ are defined in (\ref{eq3.2}). For example, if 
$$C=\begin{pmatrix}
     0 &1&0    \\
     0 &0&0    \\ 
     0 &0&z_3    \\
\end{pmatrix},$$ then the fact that $f_2=0$ implies that $z_3=1$. Thus we recover the 
 $\upomega$-Lie algebra $L_1$ appeared in the classification lists of complex and real 
3-dimensional  $\upomega$-Lie algebras; see \cite[Theorem 2 and Table 1]{CLZ14}. This example also means that  the 
 $\upomega$-Lie algebra $L_1$, defined by
 $$[x,y]=y,[x,z]=0,[y,z]=z\textrm{ and } \upomega(x,y)=1,\upomega(x,z)=\upomega(y,z)=0$$
will occur in the classification list of all 3-dimensional non-Lie  $\upomega$-Lie algebras over any field $\K$ of characteristic $\neq 2$. In the algebraically closed case, we have seen in Section \ref{sec4} that 
this $\upomega$-Lie algebra is isomorphic to $D$.

Moreover, if $$C=\begin{pmatrix}
     0 &0&1    \\
     0 &0&0    \\ 
     0 &z_2&0    \\
\end{pmatrix},$$ then it follows from the fact that $f_2=0$ that $z_2=-1$ and this recovers the  $\upomega$-Lie algebra $L_2$ in \cite[Theorem 2]{CLZ14}. Using the same method, we will recover all  3-dimensional $\upomega$-Lie algebras over $\C$ and $\R$, as well as derive some new  
3-dimensional $\upomega$-Lie algebras over other fields. 

In particular, we believe that this method will be helpful in classifying all 3-dimensional non-Lie  $\upomega$-Lie algebras over the rational field and finite fields. 
Recent applications of computational ideal theory and invariant theory in constructing and classifying  algebraic structure defined by polynomial equations over finite fields can be found in \cite{CZ23b} and \cite{CZ25}.

The second remark concerns higher-dimensional $\upomega$-Lie algebras over $\K$. For example, we consider the $4$-dimensional case. By Theorem \ref{thm2.5}, we have two nonzero canonical forms for $\upomega$: 
$$J_1=\dia\{J,0\}\textrm{ and }J_2=\dia\{J,J\}.$$
Fixing a canonical form and using the method of Gr\"{o}bner basis occurred in Section \ref{sec3}, we may obtain the vanishing ideals of some special affine subvarieties of $\A_4(\K)$. These vanishing ideals can be used to construct new 
$\upomega$-Lie algebras as in the first remark above.  

Let us close this article with the following example that illustrates this idea.

\begin{exam}{\rm 
We would like to understand the affine subvariety $X_1(\K)$ of  $\A_4(\K)$ that consists of all $4$-dimensional non-Lie 
$\upomega$-Lie algebras over $\K$ that contain the $3$-dimensional $\upomega$-Lie algebra $D$ above as a subalgebra. 
Note that we have classified these $4$-dimensional $\upomega$-Lie algebras over $\C$; see \cite[Theorem 1.5]{CZ17}. 

For any field $\K$ of characteristic $\neq 2$, we may assume that 
$$[x,e]=a_1x+a_2y+a_3z+a_4e, [y,e]=b_1x+b_2y+b_3z+b_4e,\textrm{ and }[z,e]=c_1x+c_2y+c_3z+c_4e.$$
Write $J$ for the vanishing ideal of $X_1(\K)$ in the polynomial ring $\K[x_i,y_i,z_i\mid 1\leqslant i \leqslant 4]$. Using the  Gr\"{o}bner basis method, we will derive that $J$ can be expressed as the intersection of two prime ideas $\p_1$ and $\p_2$, where
\begin{eqnarray*}
\p_1 & := & \ideal{x_1, x_2 + z_3, x_4 + 1, y_2 - z_3, y_4 - 1, z_1, z_2, z_4} \\
\p_2 & := & \ideal{x_4z_3 - x_2, x_1, y_1, y_2 - z_3, y_3, y_4 - 1, z_1, z_2, z_4}. 
\end{eqnarray*}
Hence, $X_1(\K)=V(\p_1)\cup V(\p_2)$ is a 4-dimensional affine variety consisting of two irreducible components. In particular, for $X_1(\C)$,  by setting generators of $\p_1$ or $\p_2$ to zero, we will recover those $4$-dimensional non-Lie 
$\upomega$-Lie algebras over $\C$ appeared in \cite[Theorem 1.5 (1)]{CZ17}.
\hbo}\end{exam}

\vspace{2mm}
\noindent \textbf{Acknowledgements}. 
This research was partially supported by  NNSF of China (Grant No. 12561003).


\raggedright
\end{document}